\documentclass{article}

\usepackage[utf8]{inputenc}
\usepackage[english]{babel}
\usepackage{amsfonts}
\usepackage{amssymb}
\usepackage{amsmath}
\usepackage{xypic}
\usepackage[all]{xy}
\usepackage{mathrsfs}
\usepackage{latexsym}
\usepackage{amscd}

\usepackage[mathscr]{eucal}

\usepackage{amsthm}

\usepackage{setspace}                    % Line spacing

\usepackage{tocbibind}

\usepackage{hyperref}

\newtheorem{prop}{Proposition}[section]

\newtheorem{theorem}[prop]{Theorem}

\newtheorem{lemma}[prop]{Lemma}

\newtheorem{conjecture}[prop]{Conjecture}

\theoremstyle{definition}

\newtheorem{definition}[prop]{Definition}

\theoremstyle{remark}

\newtheorem{remark}[prop]{Remark}

\newtheorem{question}[prop]{Question}

\usepackage{tikz}
\usetikzlibrary{arrows}

\newcommand{\PP}{\mathbb{P}}

\newcommand{\oo}{\mathcal{O}}

\newcommand{\rk}{\mathrm{rk}}

\newcommand{\CC}{\mathbb{C}}

\newcommand{\lls}[1]{|#1|}

\title{On linear stability and syzygy stability for 
rank 2 linear series}
\author{Abel Castorena, Ernesto C. Mistretta, Hugo Torres-L\'opez}
\begin{document}

\maketitle

%\epigraph{\itshape Think you're escaping and run into yourself. Longest way round is the shortest way home.}{---James Joyce, \textit{Uulysses - Nausicaa}} 
%

\begin{abstract}

In previous works,
the authors investigated the relationships between
linear stability of a generated linear series $|V|$ on a curve $C$, 
and slope stabillity of the vector bundle 
$M_{V,L} := \ker (V \otimes \oo_C \to L)$. 
In particular,
the second named author and L. Stoppino conjecture that, 
for a complete linear system $|L|$,
linear (semi)stability is equivalent 
to slope (semi)stability of $M_V$,
and the first and third named authors proved that this conjecture 
holds for hyperelliptic and for generic curves.

In this work we provide a counterexample to this conjecture 
on any smooth plane curve of degree $7$.

\end{abstract}

\section{Introduction}

Let $C$ be a smooth projective curve over $\CC$,
and let $L$ be a globally generated line bundle on $C$,
with $\deg L =d$ and $\rk L = h^0(L) -1$, 
let $V \subseteq H^0(C, L)$ a subspace of dimension $r+1$.
Then $M_{V,L} := \ker (V \otimes \oo_C \to L) $ 
is a rank $r$ vector bundle,
it appears in different ways and has been given different names
in the literature (cf. \cite{EinLaza},\cite{but}, \cite{mistretta06}, \cite{mistretta}, \cite{ecmistable}
\cite{leticia}, \cite{hugoleticia}).

Slope semistability of $M_{V,L}$ for a generic linear subsystem 
of a generic generated line bundle on a generic curve was conjectured by Butler in (cf. \cite{but})
 and proven in (cf. \cite{leticia}).
An analogue conjecture is still open for higher rank vector bundles.

In \cite{mistrettastoppino} the second named author and L. Stoppino
 investigate the realtionships
between linear (semi)stability 
of the linear series $|V| \subseteq |L|$,
and slope (semi)stability of $M_{V, L}$. In particular,
it is immediate to 
show that slope (semi)stability of
$M_{V, L}$ implies linear (semi)stability of $|V|$ (cf. Lemma \ref{linslope} below),
and they prove that the two conditions are equivalent 
in some cases and give some examples when they aren't.
Furthermore they conjecture that for complete linear systems,
the two conditions are always equivalent.
The only evidence on that conjecture lied in the fact that the conjecture seemed likely to hold on general curves.

In fact the first and third named authors 
showed in \cite{castorenatorres}
that the conjecture holds when $C$ is a hyperelliptic curve or a 
Brill-Noether-Petri general curve.

The purpose of this work is to show that this conjecture does not hold 
for smooth plane curves. In particular, we prove the following 

%provide a counterexample
%on any smooth plane curve of degree $7$, showing that 
%a generic, generated line bundle of degree 15 and rank 2 is linearly 
%stable but its syzygy bundle $M_L$ is not semistable.
%

\begin{theorem}

Let $C$ be a smooth plane curve of degree $d=7$.
Then a generic element $L$ in any component of $W^2_{15} (C)$ satisfies:

\begin{enumerate}

\item The complete linear series $|L|$ is linearly stable.

\item The vector bundle  $M_L$ is not semistable.
\end{enumerate}  

\end{theorem}

\subsection{Acknowledgements}

This research was partially funded by the PRIN research project
``Geometria delle Variet\`a algebriche'' code 2015EYPTSB-PE1, and partially funded by the research project
SID 2018 - NOVELLI ``Vector Bundles, Tropicalization, Fano Manifolds''.

\section{Notations and previous results}

Throughout this work, $C$ will denote a smooth projective curve of 
genus $g 
%= g(C) 
\geqslant 2$ over 
%an algebraically closed field $k$.
the field of complex numbers.
We will denote $\gamma = \gamma(C)$ the gonality of $C$,
\emph{i.e.} the smallest integer $\nu$ such that there exist a 
degree $\nu$ map to $\PP^1$. 

We will denote $\mathrm{Pic}(C)$ the Picard group of line bundles with tensor product,
and $\mathrm{Pic}^d(C)$ those of degree $d$.
For $L \in \mathrm{Pic}(C)$ we wil denote  
$|L| = \PP(H^0(C, L)^*)$ the complete linear series of effective divisors linearly equivalent to $L$,
and $|V|$ with $V \subset H^0(C, L)$ a linear series.
If needed we will use divisors (up to linear equivalence) 
and the additive notation instead of line bundles,
writing $|D|$ for $|\oo_C(D)|$,
$h^0(D)$ for $h^0(C, \oo_C(D))$, and so on.
We will denote $\omega_C$ the canonical line bundle and $K_C$ a canonical divisor.

We will denote as usual the Brill-Noether loci by
\[
W^r_d(C) = \{  P \in \mathrm{Pic}^d(C) ~|~ h^0(C, P) \geqslant r+1 \} ~.
\]
When the expected dimension of $W^r_d(C)$ is greater than  or equal to $0$,  this locus is non-empty and
every component of such a locus has dimension  greater than or equal to this expected dimension,
which is the Brill-Noether number
\[
\rho(d, r, g) = g - (r+1)(g-d+r) ~.
\]

Let $E$ be a vector bundle on $C$, 
the slope of $E$ is 
\[
\mu(E) := \frac{\deg E}{\rk E} ~.
\]

\begin{definition}

We say that $E$ is  \emph{stable} (respectively  \emph{semistable}) 
if any subbundle $0 \neq F \subsetneq E$ satisfies 
\[
\mu(F) < \mu(E) ~ 
~\textrm{ (respectively $\mu(F) \leqslant \mu(E)$)} ~.
\]

\end{definition}

Let $L$ be a line bundle on $C$, and let $V \subseteq H^0(C, L)$ 
a linear subspace. We say that the linear series $|V|$ generates 
the line bundle $L^{\prime}$, if $L^{\prime}$ is the image
of the map $V \otimes \oo_C \to L$. The rank of a linear 
series $|V|$ is $\rk |V| = \dim V -1$. 
We denote $M_{V,L} := \ker (V \otimes \oo_C \to L) $
and $M_{L} :=  M_{H^0(L),L}$. 
If $V$ generates $L$ then 
\[
\mu(M_{V,L}) = - \frac{\deg L}{\rk |V|} 
\]

\begin{definition}

We say that a pair $(L, V)$, where $L$ is a line bundle and $V \subseteq H^0(C, L)$, is a 
\emph{generating pair} if $V$ genrates $L$.
We say the linear series $|V|$,
or the generating pair $(L, V)$,
is \emph{linearly stable} (respectively \emph{linearly semistable}) if
for any subspace $W \subsetneq V$ with $\dim W \geqslant 2$, the line bundle 
$L^{\prime}$ generated by $W$ satisfies
\[
\frac{\deg L^{\prime}}{\rk |W| } >  \frac{\deg L}{\rk |V| }~
~ \textrm{ (respectively $\frac{\deg L^{\prime}}{\rk |W| } \geqslant \frac{\deg L}{\rk |V| }$)} ~.
\]

\end{definition}

\begin{lemma}
\label{linslope}
If a generating pair $(L,V)$ is such that $M_{V,L}$ is (semi)stable, then $(L,V)$ is linearly (semi)stable.
\end{lemma}

\begin{proof}
In fact for any subspace $0 \neq W \subsetneq V$ generating $L^{\prime} \subseteq L$ we have the following diagram:

\begin{equation}\label{diagramalineal}
   \xymatrix{0  \ar[r]  & M_{W,{L^{\prime}}}   \ar@{^{(}->}[d]  \ar[r]    & 
  W \otimes\mathcal{O}_C\ar@{^{(}->}[d]  \ar[r]  & {\tilde L^{\prime}}   
\ar@{^{(}->}[d] \ar[r]  & 0\\
                                0\ar[r]^{}      & M_{V,L} \ar[r]_{}& V\otimes\mathcal{O}_C  \ar[r]_{}   & 
L \ar[r]^{} & 0    ~,\\  }
                       \end{equation}
where $\mu(M_{V,L}) = - \frac{\deg L}{\rk |V|}$ and $\mu(M_{W,L^{\prime}}) = - \frac{\deg L^{\prime}}{\rk |W|}$,
therefore (semi)stability of $M_{V,L}$ implies linear (semi)stability of $(L,V)$.

\end{proof}

The second named author and L. Stoppino investigate the reverse implication of the above Lemma
in \cite{mistrettastoppino},
in particular they conjecture that equivalence holds in the following cases:

\begin{conjecture}
\label{cliffconj}
If $(L, V)$ is a generating pair on $C$, such that 
\[
\deg L \leqslant 2 ~ \rk |V| + \mathrm{Cliff}(C)
\]
then linear (semi)stability of $(L, V)$ is equivalent to (semi)stability of $M_{V,L}$.

\end{conjecture}

In the same work, the authors prove that the conjecture above holds 
if $V= H^0(C, L)$ is a complete linear series,
and they apply this to prove stability of some syszygy bundle $M_L$.
Furthermore they observe that in general the equivalence does not hold,
and observing the counter examples they can construct,
they state  the stronger conjecture:

\begin{conjecture}
\label{gonconj}
If $(L, V)$ is a generating pair on $C$, such that 
\[
\deg L \leqslant \gamma (C) \cdot \rk |V|
\]
then linear (semi)stability of $(L, V)$ is equivalent to (semi)stability of $M_{V,L}$.

\end{conjecture}

Concerning complete linear series, they conjecture that equivalence always holds in this case.

\begin{conjecture}
\label{complconj}
If $V= H^0(C, L)$ is a complete base point free linear series,  
then linear (semi)stability of the  linear series $|L|$ is equivalent to (semi)stability of $M_{L}$.

\end{conjecture}

The first and third named author proved in \cite{castorenatorres} that Conjecture 
\ref{complconj} does hold in the two opposite cases: 
when $C$ is a hyperelliptic curve and when $C$ is a Brill-Noether-Petri general curve.

The aim of this work is to show that Conjecture \ref{complconj} does not hold in general:
we can give counterexamples on any smooth plane curve of degree $7$ (so genus $15$ and gonality $6$). In particular
we will show that  on such a curve $C$
a generic line bundle in any component of the (non-empty) Brill-Noether locus $W^2_{15}(C)$ 
is globally generated, has $h^0(C,L)=3$, is linearly stable, 
but $M_{L}$ is a rank 2 vector bundle which is not semistable, hence not stable.

These constructions do not contradict Conjecture \ref{gonconj} however,
as they provide line bundles $L$ with $\deg L = 15 > \gamma (C) \cdot \rk |V| = 12$.

The techniques used follow closely C. Voisin's work in \cite{voisin}.

\section{Rank 2 linear series on higher gonality curves}

In this section we construct rank $2$ complete linear series on curves with high gonality,
and show that they are linearly stable. Most of the results we will make use of are proven in 
Voisin's work \cite{voisin}, we give details on the constructions provided there
for clarity and for better understanding of the following proofs. 

The main results are obtained as consequences of the following lemmas, to be found in \cite{mumford} and \cite{keem}:

\begin{lemma}[Mumford]
\label{mumfordlemma}

Let $C$ be a  non-hyperelliptic curve of genus $g \geqslant 4$.
Let $d,r$ be two integers such that $2 \leqslant d \leqslant g-2$ and $0 < 2r \leqslant d$.
If $\dim W^r_d (C) \geqslant d -2r -1$ then $C$ is trigonal or bielliptic or a smooth plane quintic.
In particular $\gamma(C) \leqslant 4$.

\end{lemma}

\begin{lemma}[Keem]
\label{keemlemma}

Let $C$ be a   curve of genus $g \geqslant 11$.
Let $d,r$ be two integers such that $4 \leqslant d \leqslant g +r-4$ and $r>0$.
If $\dim W^r_d (C) \geqslant d -2r -2 \geqslant 0$ then  $\gamma(C) \leqslant 4$.

\end{lemma}

\begin{theorem}[Voisin]
\label{voisinthm}

Let $C$ be a  curve of genus $g \geqslant 11$ and gonality $\gamma \geqslant 5$.
Then a generic element $D$ in any component of $W^1_{g-2} (C)$ satisfies:
\begin{enumerate}

\item $|D|$ is base point free and $h^0(D) =2$.

\item $|K_C -D|$ is base point free and $h^0(K_C -D) =3$.

\end{enumerate}  

\end{theorem}

\begin{proof}

The expected dimension of $W^1_{g-2}(C)$ is $\rho(g-2, 1, g) = g-6$, and according to Mumford's 
Lemma \ref{mumfordlemma} above, as the gonality of the curve is $\gamma \geqslant 5$,
no component has dimension greater than or equal to $g-5$.

Applying Keem's Lemma \ref{keemlemma} above, 
$\dim W^1_{g-3}(C) \leqslant g-8$. Therefore if we denote
\[
W^1_{g-3} + (C) := \{ P \in \mathrm{Pic}^{g-2}(C) ~|~ P = Q \otimes \oo_C(p) ~,~ Q \in W^1_{g-3}(C) ~,~ p\in C   \}
\]
then $W^1_{g-3} + (C)$ has empty interior in 
$W^1_{g-2}(C)$ by a dimensional count.

Now an element $L \in W^1_{g-2}(C)$ lies in $W^1_{g-3} + (C)$ if either it has 
$h^0(L) \geqslant 3$ or it has a base point. Therefore a general element of a component 
of $W^1_{g-2}(C)$ is base point free and has $h^0(L) =2$, and this proves the first point.

In order to prove the second point, we  prove the following

\noindent
\textbf{Claim:} every component of $W^2_{g-1}(C)$ has dimension at most $g-8$.

If the Claim holds, then we can proceed as above: set
\[
W^2_{g-1} + (C) := \{ P \in \mathrm{Pic}^{g}(C) ~|~ P = Q \otimes \oo_C(p) ~,~ Q \in W^2_{g-1}(C) ~,~ p\in C   \} ~,
\]
then, as every component of $W^2_g (C)$ has dimension at least 
$\rho (g,2,g) = g-6$,
$W^2_{g-1} + (C)$ has empty interior in $W^2_g (C)$ as it has smaller dimension.
Therefore a generic point  $F$ in a component of $W^2_g(C)$ does not lie 
in $W^2_{g-1} + (C)$ so it is base point free and has rank $2$.

As the application $P \mapsto P^* \otimes \omega_C$ 
is an isomorphism $\mathrm{Pic}^g(C) \to \mathrm{Pic}^{g-2}(C)$ which 
restricts to an isomorphism $ W^2_g(C) \to W^1_{g-2} (C)$,
then a generic point in a component of $W^1_{g-2}(C)$ correspond
to a generic point in a component of 
$W^2_g (C)$ and the second point is proven.

\noindent
\emph{Proof of the Claim.} By contradiction, suppose there is a component $X$ of 
$W^2_{g-1}(C)$ of dimension greater than or equal to $g-7$. Then a general element of such a component 
does not lie in $W^2_{g-2} + (C)$, as by Keem's Lemma 
(Lemma \ref{keemlemma} above)
$\dim W^2_{g-2} \leqslant g-9$.
Therefore a generic $L \in X$ is base point free and has $h^0(C, L) =3$.
For the same argument,
as $L^* \otimes \omega_C$ varies in a component of $W^2_{g-1}$ of dimension 
at least $g-7$ as well, then for generic such $L$ we have
 $L^* \otimes \omega_C$ base point free and with $h^0(C, L^* \otimes \omega_C) =3$.

Now, let us consider the morphism $\varphi_L \colon C \to \PP^2$ induced by $L$. It cannot be an immersion, otherwise its image would be a plane curve 
of degree $g-1$ and genus $g$, which is impossible for $g\geqslant 11$.
Therefore there is a couple $(p,q) \in C^2$ such that
$H^0(C, L(-p -q)) = H^0(C, L(-p)) = H^0(C, L(-q)) \cong \CC^2$.
In the following, let us use divisors.
Let us choose a divisor $\Delta \in |L|$ and let us 
denote $\Delta^{\prime} := \Delta -p -q$,
then $\oo_C({\Delta}^{\prime}) = L(-p -q) \in W^1_{g-3}$. 

Furthermore,
as both $|\Delta|$ and $|K_C - \Delta|$ are base point free,
we can see that $|K_C - \Delta^{\prime}|= |K_C - \Delta +p+q|$ is base point free as well:
in fact the only base points could be $p, q \in C$,
however, as  $h^0(C, L(-p -q)) = h^0(C, L(-p)) = h^0(C, L(-q)) =2$,
applying Riemann-Roch we have that 
$h^0(K_C - \Delta +p) = h^0(K_C - \Delta +q) =3$,
and $h^0(K_C - \Delta +p+q) =4$. 
Therefore we have 
a base point free linear system $K_C - \Delta^{\prime}$,
let us consider the induced map: 
\[
\varphi_{K_C - \Delta^{\prime}} \colon 
C \to \PP^3 ~	,
\]
then either this map is birational, and therefore there is
a finite number of couples $(x,y) \in C^2$ such that
\[
h^0(K_C - \Delta^{\prime} -x-y) = 
h^0( K_C - \Delta^{\prime} -x) =  3
\]
and so 
\[
h^0( \Delta^{\prime} +x+y) = 3 = h^0({\Delta}^{\prime}) +1 ~;
\]
or the map is a degree $m$ morphism
\[
\varphi_{K_C - \Delta^{\prime}} \colon 
C \to \overline{C} \subset \PP^3 ~	,
\]
and in this case there are infinite couples 
$(x,y)$ satisfying 
\(
h^0( \Delta^{\prime} +x+y) = 3 = h^0({\Delta}^{\prime}) +1 
\),
these are all the couples contained in the fibers of 
$\varphi_{K_C - \Delta^{\prime}}$. However,
as $|K_C - \Delta^{\prime} -p -q| = |K_C - \Delta|$ is base point free,
we see
that such fibers cannot contain more than two points
(counted with multiplicity),
so the degree of  
$\varphi_{K_C - \Delta^{\prime}} $ must be $m=2$ if it is not a birational map,
and the set of such couples $(x,y)$ has dimension $1$.

Now, let us consider the following scheme:
\[
Y := \{ (\Delta^{\prime}, x, y)  ~|~  |\Delta^{\prime}| \textrm{ is a complete } g^1_{g-3} ~,~
|\Delta^{\prime} + x + y| \textrm{ and } |K_C - \Delta^{\prime} - x - y| 
\]
\[
\textrm{ are complete and base point free } g^2_{g-1} \} ~.
\]
According to the argument above,
there must be a component $Y_0$ dominating $X \subset W^2_{g-1}(C)$,
and the fibers of $pr_1 \colon  Y_0 \to W^1_{g-3}$ have dimension at most 1.
Now as $\dim X \geqslant g-7$ by hypothesis, and $\dim W^1_{g-3} \leqslant g-8$ by
Lemma \ref{keemlemma}, then we deduce that 
$pr_1(Y_0)$ must contain an open dense subset $U_0$ of a $(g-8)$-dimensional 
 component $W \subset W^1_{g-3}$,
and the generic fibers of $pr_1$ must be $1$-dimensional.

According to the description above, 
for 
$\Delta^{\prime} \in U_0$  the residual linear series
$K_C - \Delta^{\prime} $
 is a complete and base point free 
$g^3_{g+1}$ on $C$,
inducing a degree $2$ map 
\[
\varphi_{K_C - \Delta^{\prime}} \colon
C \to \overline{C} \subset \PP^3 ~.
\]

Let us consider the normalization 
$\nu \colon \widetilde{C} \to \overline{C}$,
and the map 
$\widetilde{\varphi} \colon C \to 
\widetilde{C}$. Then 
$\widetilde{C}$ has a complete linear series of rank 3 and degree
$\frac{g+1}{2}$ therefore it is not a rational curve.
Then when $\Delta^{\prime}$ varies in $U_0$ the curve 
$\widetilde{C}$ and the map $\widetilde{\varphi}$ are fixed.

Now for a generic point $K_C - \Delta$ in $K_C - X$,
we have $K_C - \Delta = K_C - (\Delta^{\prime} + x + y)$,
where $x,y$ are points in a fiber of $\widetilde{\varphi}$. 
Therefore $K_C - \Delta$ is the pull back through 
$\widetilde{\varphi}$ of a $g^2_{(g-1)/2}$ on $\widetilde{C}$.
So we have 
\[
\dim W^2_{(g-1)/2}(\widetilde{C}) \geqslant \dim X \geqslant g-7 ~.
\] 
By the Riemann-Hurwitz formula, 
the genus $g^{\prime} = g(\widetilde{C})$ satisfies
$g^{\prime} \leqslant (g+1)/2$. 
So  we  have the following inequalities:
\[
g-7 \leqslant g^{\prime} \leqslant (g+1)/2
\]
and therefore we have that $g$ is odd and satisfies 
$11\leqslant g \leqslant 15$ and $\widetilde{C}$ is a curve 
of genus $g^{\prime}$ such that 
 $g-7 \leqslant g^{\prime} \leqslant (g+1)/2$
 and that $\dim W^2_{(g-1)/2} (\widetilde{C}) \geqslant g-7$.

We can show that these inequalities cannot hold,
and so we have proven the claim. In fact,
according to the inequalities above,
we have the following cases:

\begin{enumerate}

\item if $g=15$ then $g^{\prime} =8$ and
$\dim W^2_{7} (\widetilde{C}) \geqslant 8$ which is impossible,
as it would imply that $ W^2_{7} (\widetilde{C}) = \mathrm{Pic}^7(\widetilde{C})$;
\item if $g =13$ then $g^{\prime} =7$ or $g^{\prime} =6$ and
$\dim W^2_{6} (\widetilde{C}) \geqslant 6$ which is impossible,
as  the case $g^{\prime} =6$  would imply that 
$ W^2_{6} (\widetilde{C}) = \mathrm{Pic}^6(\widetilde{C})$, and the case $g^{\prime} =7$
would imply that $ W^2_{6} (\widetilde{C}) \subseteq \mathrm{Pic}^6(\widetilde{C})$
has codimension at most 1 and is therefore equal to the theta divisor;
\item  if $g =11$ then $g^{\prime} =6$ or $g^{\prime} =5$ 
 or $g^{\prime} =4$ and
$\dim W^2_{5} (\widetilde{C}) \geqslant 4$ which is impossible for similar arguments.
\end{enumerate}

This completes the proof of the Claim and therefore of the Theorem.

\end{proof}

\begin{lemma}
\label{lemmamult}

Let $V \subseteq H^0(C, L)$ be a rank-$2$  linear series on $C$,
generating $L$, such that 
\[
\varphi_{|V|} \colon C \to \overline{C} \subset \PP^2
\]
is a birational morphism.

Then $|V|$ is linearly (semi)stable if and only if  all points $p \in \overline{C}$ 
have  multiplicity  $m_p(\overline{C}) < \deg L / 2$
(or $m_p(\overline{C}) \leqslant \deg L / 2$
for semistability).

\end{lemma}

\begin{proof}

Since $\varphi$ is a birational morphism,
then $\deg L = \deg \varphi^* \oo_{\PP^2}(1) = \deg \overline{C}$.
The subspace $V \subseteq H^0(C, L)$ has dimension 3, and a 2-dimensional subspace 
$W \subset V$ corresponds to a point in $ q \in \PP^2 = \PP (V^*)$, and induces a projection map from this point:
\[
\pi \colon \PP(V^*) \setminus \{q\} \to \PP(W^*) \cong \PP^1 ~.
\]
This map composed with $\varphi_{|V|}$ and extended to $C$ is induced by the rank 1 linear subseries of $W \subset V$,
which  generates a line bundle $L^{\prime} \subseteq L$ of degree $\deg \overline{C} - m_q(\overline{C})$,
where $m_q (\overline{C}) =0$ if $q \notin \overline{C}$.
Therefore we have linear stability if and only if 
$\deg L - m_q(C) > \deg L /2$ for all $p \in \overline{C}$. Same for semistability.

\end{proof}

\begin{theorem}
\label{voisinconsequence}

Let $C$ be a  curve of genus $g \geqslant 11$ and gonality $\gamma \geqslant 5$.
Then a generic element $L$ in any component of $W^2_{g} (C)$ satisfies:

\begin{enumerate}

\item
\label{puntouno} The complete linear series $|L|$ is base point free and $h^0(C,L) =3$.

\item
\label{puntodue} The complete linear series $|L|$ induces a birational morphism
\(
\varphi_{L} \colon C \to \overline{C} \subset \PP^2
\),
where $\overline{C} \subset \PP^2$ is a singular curve of degree $g$.

\item
\label{puntotre} The complete linear series  $|L|$ is linearly stable.

\end{enumerate}

\end{theorem}

\begin{proof}

The first point  follows from 
Theorem \ref{voisinthm} above. We have to show that such a generic element 
induces a birational map to its image and is linearly stable.

Let us prove that for a divisor $D \in W^2_{g}$ generic in an irreducible component of
$W^2_{g}$, the linear series $| D|$ induces a birational morphism 
$\varphi_{D} \colon C \to \overline{C} \subset \PP^2$.
Let us observe that by the first point the linear series 
$\lls{D}$ and $|K_C - D|$ are complete and base point free $g^2_g$
and
$g^1_{g-2}$.

As $\varphi_{D}$ cannot be an embedding, then there exist $p, q \in C$ such that 
$H^0(D -p ) = H^0(D - q )
 =H^0(D -p -q) \cong \CC^2$.  
 So the divisor $D^{\prime} := D - p - q$ satisfies:

\begin{enumerate}

\item $|D^{\prime}|$ is a complete $g^1_{g-2}$;
\item $|K_C - D^{\prime}|$ is a complete base point free $g^2_{g}$;
\item $|D^{\prime} +p +q|$ is a complete base point free $g^2_{g}$;
\item $|K_C - D^{\prime} -p -q|$ is a complete base point free $g^1_{g-2}$.
\end{enumerate}
Furthermore, two points $p,q \in C$ satisfy 
$\varphi_{D}(p) = \varphi_{D}(q)$ if and only if 
the divisor $D^{\prime} = D - p - q$ satisfies the  conditions above.

Now let us consider the following scheme:
\[
Y := \{ (D^{\prime}, x, y)  ~|~  | D^{\prime}| \textrm{ is a complete } g^1_{g-2} ~,~
\]
\[
| D^{\prime} +x +y|  \textrm{ is a complete and base point free } 
g^2_{g}  ~,
\]
\[
| K_C - D^{\prime} -x -y|  \textrm{ is a complete and base point free } g^1_{g-2} \} ~,
\]
and the two maps: 
\[pr_1 \colon Y \to W^1_{g-2}(C) \textrm{ and  }
\pi \colon Y \to W^2_{g}(C) , (D^{\prime}, x, y) \mapsto D^{\prime} + x + y ~.
\]
According to the description above,
the morphism $\pi$ is dominant on every component of $W^2_{g}(C)$,
and the fiber of the morphism $\pi$ over a generic divisor $D$ in a component of $W^2_{g}(C)$,
is the set of all triples $(D-x-y, x, y)$ such that 
$\varphi_{D}(x) = \varphi_{D}(y)$.
Remark that the fiber over a divisor 
$D^{\prime} \in \mathrm{Im} (pr_1)$ of 
$pr_1 \colon Y \to W^2_g(C)$ is 
the set of all triples $(D^{\prime},x,y)$ such that 
$\varphi_{K_C-D^{\prime}} (x) = \varphi_{K_C -D^{\prime}} (y)$.

Therefore, in order to prove point (\ref{puntodue}) in the statement of the theorem,
let us suppose by contradiction that for a generic divisor $D$ in a component $X$ of $W^2_{g}(C)$
the morphism $\varphi_{D} \colon C \to \PP^2$ is not birational to its image.
Then the fibers of $\pi$ have positive dimension, so there is a component $Y_0 \subseteq Y$ 
such that $\dim Y_0 > \dim X = g-6$.
That component $Y_0$ must be dominant through $pr_1$ onto a component $W_0$ of $W^1_{g-2}(C)$ as well,
otherwise the generic fiber of $pr_1$ would have dimension  $2$ which is impossible.
Then we deduce that for a generic element $D^{\prime} \in W_0$,
the morphism $\varphi_{K_C - D^{\prime}} \colon C \to \overline(C) \subset \PP^2$
is not birational, and has degree $m >1$.
With the same argument as in the proof of Theorem \ref{voisinthm},
as we know that  $|K_C - D^{\prime} -x -y|$ is a complete base point free $g^1_{g-2}$ for some $x,y \in C$,
we see that 
in fact it must be of degree $2$ in this case.

Then we proceed as in the proof of Theorem \ref{voisinthm},
considering the normalization 
$\nu \colon \widetilde{C} \to \overline{C}$,
and the map 
$\widetilde{\varphi} \colon C \to 
\widetilde{C}$. 
Then 
$\widetilde{C}$ has a complete linear series of rank 3 and degree
$\frac{g}{2}$ therefore it is not a rational curve.
And so when $D^{\prime}$ varies in $W_0$ the curve 
$\widetilde{C}$ and the map $\widetilde{\varphi}$ are fixed.
Let us call $g^{\prime} = g(\widetilde{C})$ its genus.

The divisor  $K_C - D^{\prime}$
is the pull back through 
$\widetilde{\varphi}$ of a $g^2_{g/2}$ on $\widetilde{C}$.
So we have 
\[
(g+1)/2 \geqslant g^{\prime} \geqslant \dim W^2_{g/2}(\widetilde{C}) \geqslant \dim W_0 = g-6 ~,
\] 
the first inequality following from Riemann-Hurwitz formula.
As $g \geqslant 11$ by hypothesis, and even,
then we must have $g=12$, $g^{\prime} = g/2 = 6$, 
and $\dim W^2_{g/2}(\widetilde{C}) = g-6 = g^{\prime}$ which is impossible.
So this completes the proof of point (\ref{puntodue}).

So we have proven that for a divisor
$D$ generic in a component of $W^2_{g}$ 
the liear series $\lls{D}$ and $\lls{K_C -D}$ are base point free, and
the map 
$\varphi_{D} \colon C \to \overline{C} \subset \PP^2$ is birational.
Let us prove that the multiplicity of any point $p \in \overline{C}$ is at most $2$.

Consider the scheme defined above:

\[
Y := \{ (D^{\prime}, x, y)  ~|~  | D^{\prime}| \textrm{ is a complete } g^1_{g-2} ~,~
\]
\[
| D^{\prime} +x +y|  \textrm{ is a complete and base point free } 
g^2_{g}  ~,
\]
\[
| K_C - D^{\prime} -x -y|  \textrm{ is a complete and base point free } g^1_{g-2} \} ~,
\]
and the two maps: 
\[pr_1 \colon Y \to W^1_{g-2}(C) \textrm{ and  }
\pi \colon Y \to W^2_{g}(C) , (D^{\prime}, x, y) \mapsto D^{\prime} + x + y ~.
\]

{\bf We claim that:} every component of $Y$ 
dominating a component of $W^2_g$ through $\pi$,
dominates through $pr_1$ a component of $W^1_{g-2}$ as well.

If the claim holds,
we can show that the multiplicity $m_p(\overline{C})$,
of any point of $p \in \overline{C} = \varphi_D (C) \subset \PP^2$,
is at most $2$.
In fact in that case, for $D$ generic in a component of $W^2_{g}$,
we have that for any couple $(x, y) \in C^2$ 
such that $H^0(D-x) = H^0(D-y) = H^0(D-x-y)$,
the divisor $D^{\prime} = D -x -y$ is generic in a component of
$W^1_{g-2}$ so it is a base point free divisor, therefore the image 
$\varphi_D(x) = \varphi_D (y)$ cannot have multiplicity higher than $2$.

Let us prove now that the claim above holds:
we have to show that any component $Y_0$ of $Y$ that dominates 
through $\pi$
a component of $W^2_g$ dominates a component of 
$W^1_{g-2}$ through $p_1$.

Recall that by Mumford's lemma \ref{mumfordlemma} all components of 
$W^1_{g-2}$ have dimension  $g-6$ (hence all components of $W^2_g$ as well).

Suppose by contradicion that there is a component 
 $Y_0$ of $Y$, dominating a component of $W^2_g$,
 wich does not dominate any component in $W^1_{g-2}$.
Then the component $Y_0$ has dimension $g-6$,
as the fiber $\pi^{-1}(D)$ of a generic divisor
$D \in W^2_g$ is finite.
The 
image $p_1(Y_0)$ is then a locus
strictly contained in a component of $W^1_{g-2}$,
and therefore of dimension $g-7$,
as the fibers cannot have dimension greater than $1$.

Now, for a given $D^{\prime} \in pr_1(Y_0)$,
the divisor $K_C - D^{\prime}$ is base point free and induces 
a map 
$\varphi_{K_C - D^{\prime}} \colon C \to \overline{C} \subset \PP^2$.
As the fiber 
$pr_1^{-1}(D^{\prime})$ is positive dimensional,
the morphism $\varphi_{K_C - D^{\prime}}$ 
is not birational,
and as 
$| K_C - D^{\prime} -x -y|$  
is complete and base point free,
then
$\varphi_{K_C - D^{\prime}}$ is a degree 2 morphism.

Proceeding as in the previous proofs,
we see that the normalization 
$\widetilde{C}$ 
of $\overline{C}$
and
the morphism $\widetilde{\varphi} \colon C \to \widetilde{C}$ 
do not vary
when $D^{\prime}$ varies in $pr_1(Y_0)$,
and that the divisor $K - D^{\prime}$ 
is the pull back $\widetilde{\varphi}^*{E}$
of a divisor $E \in W^2_{g/2}(\widetilde{C})$.

Let us call $g^{\prime} = g(\widetilde{C})$ 
the genus of $\widetilde{C}$,
then we have the following inequalities:
\[ 
(g+1)/2 \geqslant g^{\prime} \geqslant \dim W^2_{g/2}(\widetilde{C}) 
\geqslant g-7 ~.\]

Then $g$ must be even and we have the following cases:

\begin{enumerate}

\item $g=12$ and $5 \leqslant g^{\prime} \leqslant 6$;
\item $g =14$ and $g^{\prime} = 7$.

\end{enumerate}
  
The first case would have either $g^{\prime} = 5$ 
and $\dim W^2_6(\widetilde{C}) = 5$, which is impossible;
or $g^{\prime} = 6$ 
and $\dim W^2_6(\widetilde{C}) = 5$ which is impossible as well.

The second case satisfies
$g =14$,  $g^{\prime} = 7$, and $\dim W^2_6(\widetilde{C}) = 7$,
which is impossible again.

Therefore we have shown the claim that 
every component of $Y$ 
dominating a component of $W^2_g$ through $\pi$,
dominates  through $p_1$  a component of $W^1_{g-2}$ as well,
and we have seen that this implies that
for a generic $D \in W^2_g(C)$ the morphism
$\varphi_D \colon C \to \overline{C}$ is birational
and its image $\overline{C}$ has points of multiplicity at 
most $2$.
Now to complete the proof of point (\ref{puntotre}) in the theorem
we just have to apply Lemma \ref{lemmamult}.

\end{proof}

\begin{remark}

Theorem \ref{voisinconsequence}
is basically equivalent to Proposition II.0 part (b) in \cite{voisin}.
We include the proof here as it is interesting in itself,
and as the statement is not proven in that work
(it uses similar techniques as in previous results in the same article
\emph{i.e.} as in the proof of Theorem \ref{voisinthm}).

\end{remark}

\section{Counterexamples on plane curves}

In this section we show  that any smooth plane curve of degree $7$ 
admits counterexamples to Conjecture \ref{complconj}.

\begin{theorem}

Let $C$ be a smooth plane curve of degree $d=7$.
Then a generic element $L$ in any component of $W^2_{15} (C)$ satisfies:

\begin{enumerate}

\item The complete linear series $|L|$ is base point free and linearly stable.

\item The vector bundle  $M_L$ is not semistable.
\end{enumerate}  

\end{theorem}

\begin{proof}

The first point is given by Theorem \ref{voisinconsequence} above,
as the curve has genus $g=15$ and gonality $\gamma =6$. 
We have to exhibit a destabilization of $M_L$ in this case.

Let us consider the line bundle $B = \oo_{\PP^2}(1)_{|C}$, it is a line bundle
of degree $7$ with $h^0(C, B) = 3$. Using  the exact sequence
\[
0 \to M_L \otimes B \to H^0(C, L) \otimes B \to L \otimes B \to 0
\]
and passing to cohomology, we have:
\[
0 \to H^0(C, M_L \otimes B) \to H^0(C, L) \otimes H^0(C,B) \to H^0(C, L \otimes B) ~.
\]

Now, let us call $W = W^0_6 (C) \subset \mathrm{Pic}^6(C)$ the locus of effective divisors of degree $6$,
then clearly $\dim W = 6$ , so the locus 
\[
(\omega_C \otimes B^*) - W := \{L \in W^2_{15} (C) ~ | ~  L = \omega_C \otimes B^* \otimes F^* ~,~ F \in W \}
\]
has dimension $6$ as well. As every component of $W^2_{15} (C)$ has dimension at least $g-6 = 9$,
then a general line bundle $L$ in such a component is not contained in $(\omega_C) - W$.
Therefore for a general element $L$ of a component of $W^2_{15} (C)$ we have:
\[
H^0(C, \omega_C \otimes B^* \otimes L^*) \cong H^1(C, B \otimes L)^* = 0 ~,
\]
so by Riemann-Roch we have 
\[
h^0(C,  B \otimes L) = \deg (B \otimes L) + 1 - g = 8 < \dim H^0(C, L) \otimes H^0(C,B) =9 ~.
\]
Therefore $H^0(C, M_L \otimes B) \neq 0$ and we have an injection:
\[
B^* \hookrightarrow M_L
\]
which provides a destabilization as $\mu (B^*) = -7 > \mu(M_L) = -15/2$. 

\end{proof}

\begin{remark}

For a line bundle $L$ as above we have the following diagram:
\begin{equation}\label{diagramabutler}
   \xymatrix{0  \ar[r]  & B^*   \ar@{^{(}->}[d]  \ar[r]    & 
 H^0(C,B)^* \otimes\mathcal{O}_C \ar[d]^{\cong}  \ar[r]  & {F}   
\ar@{->>}[d] \ar[r]  & 0\\
                                0\ar[r]^{}      & M_{V,L} \ar[r]_{}& V\otimes\mathcal{O}_C  \ar[r]_{}   & 
L \ar[r]^{} & 0    ~,\\  }
                       \end{equation}
where $F:= \ker(H^0(C,B) \otimes \oo_C \twoheadrightarrow B)^*$ is a rank 2 vector bundle.
However we cannot have a diagram as in (\ref{diagramalineal}) which would provide a linear destabilization;
in particular, the line bundle $L$, being generic, does not admit an injection $B \hookrightarrow L$.

\end{remark}

\begin{remark}

We remark the techniques in the previous section
provide linearly stable complete base point free
linear systems of rank $2$ on all curves 
with genus $g \geqslant 11$ and gonality $\gamma \geqslant 5$,
however the very same techniques
cannot be applied to find other conterexamples on plane curves 
with different degrees.
\end{remark}

\begin{remark}

In the recent works \cite{5f, ecmist2, ecmist3} 
the first named author considers stable base loci, augmented and restricted base loci
for vector bundles. It would be interesting to compute 
explicitely the base loci in these cases 
for the  unstable bundles  $M_L^*$ constructed above. In fact these are 
globally generated vector bundles, therefore they are nef vector bundles,
however they are not semistable, and need not be ample.

\end{remark}

\begin{question}

As smooth plane curves of degree $7$ are 
not generic in the moduli space $\mathcal{M}_{15}$ 
of smooth genus $15$ curves,
and as all of them do not satisfy Conjecture \ref{complconj},
it would be interesting to describe the locus in 
$\mathcal{M}_{15}$
of all curves not satisfying Conjecture \ref{complconj}
and its numerical properties.

\end{question}

%%
%%\bibliographystyle{amsalpha223}
%%\bibliography{counterexbiblio}

\newcommand{\etalchar}[1]{$^{#1}$}
\providecommand{\bysame}{\leavevmode\hbox to3em{\hrulefill}\thinspace}
\providecommand{\MR}{\relax\ifhmode\unskip\space\fi MR }
\providecommand{\MRhref}[2]{%
  \href{http://www.ams.org/mathscinet-getitem?mr=#1}{#2}
}
\providecommand{\href}[2]{#2}

%%\providecommand{\bysame}{\leavevmode\hbox to3em{\hrulefill}\thinspace}
%%\providecommand{\MR}{\relax\ifhmode\unskip\space\fi MR }
%%% \MRhref is called by the amsart/book/proc definition of \MR.
%%\providecommand{\MRhref}[2]{%
%%  \href{http://www.ams.org/mathscinet-getitem?mr=#1}{#2}
%%}
%%\providecommand{\href}[2]{#2}
%%\begin{thebibliography}{Mum74}
%%

\end{document}